\documentclass{article}[10pt]
\usepackage{latexsym}
\usepackage{amssymb}
\usepackage{amsmath,amsfonts}
\usepackage{graphicx}
\usepackage{tikz}
\usepackage{algorithm2e}
\usepackage{algorithmic}
\usepackage{todonotes}
\usepackage{amsthm}
\usetikzlibrary{automata}
\usetikzlibrary{positioning}
\usetikzlibrary{arrows}
\usetikzlibrary{arrows.meta}
\usetikzlibrary{graphs}
\usetikzlibrary{graphs.standard}
\def \Dn{{\mathcal D}_n}

\def \Gn{{\mathcal G}_n}

\def \Tn{{\mathcal T}_n}

\newtheorem{theorem}{Theorem}[section]
\newtheorem{prop}[theorem]{Proposition}
\newtheorem{lemma}{Lemma}[section]
\newtheorem{corollary}{Corollary}[section]
\newtheorem{example}{Example}[section]
\newtheorem{definition}{Definition}[section]

\newtheorem{conjecture}[theorem]{Conjecture}

\newtheorem{question}[theorem]{Question}

\def \dd{\hfill \baseb \vskip .5cm}
\def \d{{\noindent \it Proof. } }

\def\ex{\begin{example}}
\def\eex{\end{example}}
\def\exx{\end{example}}
\def\t{\begin{theorem}}
\def\tt{\end{theorem}}
\def\D{\begin{definition}}
\def\DD{\end{definition}}
\def\l{\begin{lemma}}
\def\ll{\end{lemma}}
\def\c{\begin{corollary}}
\def\cc{\end{corollary}}
\def\cj{\begin{conjecture}}
\def\cjj{\end{conjecture}}
\def\e{\begin{equation}}
\def\ee{\end{equation}}
\def\p{\begin{prop}}
\def\pp{\end{prop}}
\def\q{\begin{question}}
\def\qq{\end{question}}

\def \Dn{{\mathcal D}_n}

\def \Gn{{\mathcal G}_n}

\def \Tn{{\mathcal T}_n}

\def \inv{^{-1}}

\newcommand{\baseb}{\hfill \rule{2mm}{2mm}} 

\begin{document}

\baselineskip 15pt

\title{{\bf (2,3)-Cordial Trees and Paths}\thanks{AMS Classification number: 05C20, 05C38, 05C78
} \thanks{Key words and phrases: orientation of an undirected graph, graph labeling, cordial labeling, (2,3)-cordial digraph
}}

\author{Manuel Santana\thanks{corresponding Author}\\ Department of Mathematics and Statistics \\Utah State University \\ Logan, Utah 84322-3900, U.S.A\\manuelarturosantana@gmail.com\\  \\  Jonathan Mousley \\ Department of Mathematics and Statistics \\Utah State University \\ Logan, Utah 84322-3900, U.S.A \\jonathanmousley@gmail.com\\ \\Dave Brown \\ Department of Mathematics and Statistics \\Utah State University \\ Logan, Utah 84322-3900, U.S.A\\ david.e.brown@usu.edu\\ \\ LeRoy B. Beasley\\ Department of Mathematics and Statistics \\Utah State University \\ Logan, Utah 84322-3900, U.S.A\\ leroy.b.beasley@aggiemail.usu.edu}
\date{}

\maketitle

\begin{abstract}  Recently L. B. Beasley introduced $(2,3)$-cordial labelings of directed graphs in \cite{B}.  He conjectured that every orientation of a path of length at least five is $(2,3)$-cordial, and that every tree of max degree $n =3$ has a cordial orientation. In this paper we formally define $(2,3)-$cordiality from the viewpoint of {\it quasigroup} cordiality. We show both conjectures to be false, discuss the $(2,3)$-cordiality of orientations of the Petersen graph, and establish an upper bound for the number of edges a graph can have and still be $(2,3)$-orientable. 
\end{abstract}

\section{Introduction}

Let $G = (V,E)$ be an undirected graph with vertex set $V$ and edge set $E$, a convention we will use throughout this paper.  A $(0,1)$-labelling of the vertex set is a mapping $f:V\to \{0,1\}$  and is  said to be {\em friendly} if approximately one half of the vertices are labelled 0 and the others labelled 1. An induced labelling of the edge set  is a mapping $g:E\to \{0,1\}$ where for an edge $uv, g(uv)= \hat{g}(f(u),f(v))$ for some $\hat{g}:\{0,1\}\times\{0,1\}\to \{0,1\}$ and is said to be cordial if $f$ is friendly and about one half the edges of $G$ are labelled 0.  A graph, $G$,  is called {\em cordial} if there exists a cordial induced labelling of the edge set of $G$.  

In this article we investigate a cordial labelling of directed graphs that is not  merely a cordial labelling of the underlying undirected graph. This labeling was introduced by L. B. Beasley in \cite{B}. Let $D=(V,A)$ be a directed graph with vertex set $V$ and arc set $A$, with a $(0,1)$ vertex set mapping $f:V\to\{0,1\}$. Let  $g:A\to \{-1,0,1\}$ be the induced  labeling of the arcs of $D$  such that for any $\overrightarrow{uv}$, by which we mean an arc going from $u$ to $v$, $g(\overrightarrow{uv}) = f(v) - f(u)$.   The digraph $D$ is said to be $(2,3)$-cordial if there exists a friendly labeling on $D$ with this induced labeling on the arc set such that approximately one third of the arcs receive each labeling. Applications of balanced graph labelings can be found in the introduction of \cite{SD}.

In this paper we will formally define $(2,3)$-cordiality starting from the view of quasi-groups, resolve two conjecture posed in \cite{B}, both in the negative, and discuss the $(2,3)$ cordiality of the Petersen graph and complete graphs.

\section{Preliminaries}
 In \cite{H} Hovey introduced $A$ cordial graphs, with vertex labeling of an abelian group $A$. We use the short hand that for any $a \in A$ that $V_a, E_a$ is the number of vertices or edges labeled $a$ respectively We repeat the definition here, and then proceed to generalize it to the directed graph case. 
\D A labeling function $f$ is said to be balanced if it is surjective, and if for all $a$ and $b$ in the image of $f$, $||f\inv(a)| - |f\inv(b)|| \le 1$. 
 
\DD
 
\D Let $A$ be an abelian group. A graph $G$  is is $A -$cordial if there is a balanced labeling $f:V \to A$ that 
induces an edge labeling $f(a,b) = f(a) + f(b)$ such that the vertex labeling and edge labeling are balanced.
\DD

In \cite{PW} Penchenik and Wise generalized this idea by introducing {\it quasi-group cordiality}. A {\it quasi-group} $Q$ is a set with binary operation $\cdot$ such that for all $a,b \in Q$ there exist unique $c,d \in Q$ such that $a \cdot c = b$ and $d \cdot a = b$. Particularly all (non-abelian) groups are quasi groups. They offer the following definition for quasigroup cordiality.
\D\label{quaigroup} Let $Q$ be a quasi-group and $G$ a directed graph. A labeling $f:V \to Q$ induces a labeling of the arcs in the following way. If $(a,b)$ is an arc with head $a$, then $f(a,b) = f(a) \cdot f(b)$. If there is a balanced vertex labeling of $G$ that induces a balanced edge labeling of $G$, then we say that $G$ is $Q-cordial$.
\DD

We extend this now to a different form of quasi-group Cordiality defined as thus.
\D Let $Q$ be a quasi-group with subset $\mathbb{Q}$ and $G$ a directed graph. A labeling $f:V \to \mathbb{Q}$ induces an arc labeling as in definition \ref{quaigroup}. If there is a balanced vertex labeling of $G$ that induces a balanced arc labeling of $G$, then we say $G$ is $(\mathbb{Q},Q)-cordial$.
\DD

 In this paper consider only the simplest case $(\mathbb{Z}_2,\mathbb{Z}_3^-)-cordial$ defined by Beasly as $(2,3)-$cordial in \cite{B}. We will offer the formal definition of $(2,3)-$cordial here.

\D A labeling $f: V \to \{0,1\}$ is said to be friendly if it is balanced. 
\DD



\D Let $\Dn$ be the set of all digraphs on $n$ vertices. We will define $\Tn$ as the subset of $\Dn$ that consists of all digon-free digraphs, where a digon is a two cycle on a digraph.\DD

\D Let $D\in \Tn$ with $D=(V,A)$. Let $f:V\to \{0,1\}$ be a friendly labeling of the vertex set $V$ of $D$. Let $g:A\to\{-1,0,1\}$ be an induced labeling of the arcs of $D$ such that for any $i,j\in\{-1,0,1\}$, $-1\leq |g^{-1}(i)|-|g^{-1}(j)| \leq 1$. Such a labeling is called a {\em$(2,3)$-cordial} labeling, and a digraph $D\in\Tn$ that can possess a $(2,3)$-cordial labeling will be called a {\em$(2,3)$-cordial} digraph.
\DD

\D Let $D=(V,A)$ be a digraph with vertex labelling $f:V\to \{0,1\}$ and with induced arc labelling $g:A\to\{-1,0,1,\}$.  Define $\Gamma_{f,g}$ to be the real triple $\Gamma_{f,g}(D)=(\alpha,\beta,\gamma)$ where $\alpha=|g^{-1}(1)|, \beta=|g^{-1}(-1)|,$ and $\gamma =|g^{-1}(0)|$.  \DD

Let $D\in\Tn$ and let $D^r$ be the digraph such that every arc of $D$ is reversed, so that $\overrightarrow{uv}$ is an arc in $D^r$ if and only if $\overrightarrow{vu}$ is an arc in $D$.  Let $f$ be a $(0,1)$-labeling of the vertices of $D$ and let $g(\overrightarrow{uv}) =f(v)-f(u)$ so that $g$ is a $(-1,0,1)$-labeling of the arcs of $D$.  Let $\overline{f}$ be the complementary $(0,1)$-labeling of the vertices of $D$, so that $\overline{f}(v)=0$ if and only if $f(v)=1$.  Let $\overline{g}$ be the corresponding induced arc labeling of $D$, $\overline{g}(\overrightarrow{uv}) =\overline{f}(v)-\overline{f}(u)$.

\l \label{lab} Let $D\in\Tn$ with vertex labeling $f$ and induced arc labeling $g$.  Let $\Gamma_{f,g}(D)=(\alpha,\beta,\gamma)$.   Then \begin{enumerate}\item $\Gamma_{f,g}(D^r)=(\beta,\alpha,\gamma)$. \item $\Gamma_{\overline{f},\overline{g}}(D)=(\beta,\alpha,\gamma)$, and \item $\Gamma_{\overline{f},\overline{g}}(D^R)=\Gamma_{f,g}(D).$\end{enumerate} \ll
\d If an arc is labeled 1, -1, 0 respectively then reversing the labeling of the incident vertices gives a labeling of -1, 1, 0 respectively,  If an arc $\overrightarrow{uv}$ is labeled 1, -1, 0 respectively, then $\overrightarrow{vu}$ would be labeled -1, 1, 0 respectively. \dd

Also in this article we will study when undirected graphs can have their arcs given an orientation such that the resulting graph is $(2,3)-$cordial. We finish this section with a couple of definitions for that.

\D Define $\Gn$ to be the set of all simple, undirected, connected graphs. We say $G \in \Gn$ has vertex set $V$ and edge set $E$ and denote it by $G = (V,E).$
\DD

\D Let $G\in\Gn$.  An {\em orientation} of $G$ is a digraph $D(G)$ whose vertex set is the same as the vertex set of $G$ and whose arc set consists of the same number of arcs as the number of edges of $G$ such that given an edge $\{u,v\}$ of $G$, either $\overrightarrow{uv}$ or $\overrightarrow{vu}$ is an arc of $D(G)$ but not both, so that $D(G)$ is digon free.  A graph $G$ is said to be {\em $(2,3)$-orientable} if there exists and orientation of $G$, $D(G)$, that  is $(2,3)$-cordial.\DD

\section{Resolution of Two Conjectures}

We begin with the first conjecture.

\noindent
\textbf{Conjecture 1}\cite[Conjecture 4.1]{B}: Every 
orientation of every path is (2,3) cordial except for a path with four vertices.

The orientation in Figure 1 of the ten path has no cordial labeling.

\begin{figure}[!h]\label{10path}
\centering
\begin{tikzpicture}
    \node[shape=circle,draw=black] (A) at (0,0) {};
    \node[shape=circle,draw=black] (B) at (1,0) {};
    \node[shape=circle,draw=black] (C) at (2,0) {};
    \node[shape=circle,draw=black] (D) at (3,0) {};
    \node[shape=circle,draw=black] (E) at (4,0) {};
    \node[shape=circle,draw=black] (F) at (5,0) {};
    \node[shape=circle,draw=black] (G) at (6,0) {};
    \node[shape=circle,draw=black] (H) at (7,0) {};
    \node[shape=circle,draw=black] (I) at (8,0) {};
    \node[shape=circle,draw=black] (J) at (9,0) {} ;
\begin{scope}[>={Stealth[black]},
              every edge/.style={draw}];
    \path [->] (A) edge node[left] {} (B);
    \path [->](C) edge node[right] {} (B);
    \path [->](C) edge node[left] {} (D);
    \path [->](C) edge node[left] {} (D);
    \path [->](E) edge node[right] {} (D);
    \path [->](E) edge node[left] {} (F);
    \path [->](G) edge node[right] {} (F);
    \path [->](G) edge node[right] {} (H);
    \path [->](I) edge node[left] {} (H);
    \path [->](I) edge node[right] {} (J);
\end{scope}
\end{tikzpicture}\\
\label{tenpath}
\caption{}
\end{figure}

We used the following brute force algorithm to test to see if a certain arc orientation is cordial on a friendly labeling of a graph.\\
\begin{centering}
\begin{algorithm}[H]
\KwData{Arcs, Verticies}
\KwResult{Determine if an orientation is cordial on a path}
\For{arc in Arcs}{
    current = first vertex\\
    next = second vertex\\
    \If{arc is left}{
    edgeLabel = current - next
    }
    \Else{
    edgeLabel = current - next}
    store edge label\\
    current = next\\
    next = next vertex\\
}
\If{edge labels are cordial}{return it is cordial}
\end{algorithm}
\end{centering}

In investigating the conjecture we had to test every possible friendly labeling and arc orientation on ten vertices. If we let $n$ denote
the number of vertices on the path, then checking every possible friendly labeling against every arc set has complexity of $O(2^k)$. As a slight optimization by  Lemma \ref{lab} with out loss of generality we can fix the first arc and the first label and still account for all cases up to isomorphism. That means there will be $2^{n-2}$ arc orientations to test. In calculating all arc orientations of the ten path we found the only orientation that is not $(2,3) -$coridal is the one in figure \ref{tenpath}.  The next known case of a non $(2,3) -$orientable path is one on 22 vertices with the same alternating arc structure.

\noindent 
\textbf{Conjecture 2}\cite[Conjecture2.3]{B}: Every tree of max degree $3$ is $(2,3)-$orientable.\\
\begin{figure}[!h]
\centering
\begin{tikzpicture}
    \node[shape=circle,draw=black] (A) at (0,0) {};
    \node[shape=circle,draw=black] (B) at (1,0) {};
    \node[shape=circle,draw=black] (C) at (2,0) {};
    \node[shape=circle,draw=black] (D) at (3,0) {};
    \node[shape=circle,draw=black] (E) at (4,0) {};
    \node[shape=circle,draw=black] (F) at (5,0) {};
    \node[shape=circle,draw=black] (G) at (1,1) {};
    \node[shape=circle,draw=black] (H) at (2,1) {};
    \node[shape=circle,draw=black] (I) at (3,1) {};
    \node[shape=circle,draw=black] (J) at (4,1) {} ;
    \path [-] (A) edge node[left] {} (B);
    \path [-](C) edge node[right] {} (B);
    \path [-](C) edge node[left] {} (D);
    \path [-](E) edge node[right] {} (D);
    \path [-](E) edge node[left] {} (F);
    \path [-](B) edge node[right] {} (G);
    \path [-](C) edge node[right] {} (H);
    \path [-](D) edge node[left] {} (I);
    \path [-](E) edge node[right] {} (J);

\end{tikzpicture}\\
\caption{}
\end{figure}
Figure 2 is a counter example. Though easily proved with a computer by a similar argument as on the one above, we will prove by cases.
\begin{proof}
First note that the value of $\gamma$ in a vertex labeling of the graph does not depend on arc orientation. Thus we show that there does not exist a friendly labeling on the vertex set of the graph that induces an edge labeling of $\gamma = 3$, by only considering vertex labelings on the undirected graph.
We will separate our graph into a right and left subgraph and then connect them
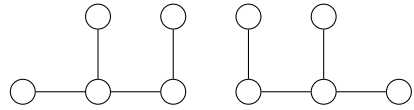
\begin{figure}[!h]
\centering
\begin{tikzpicture}
    \node[shape=circle,draw=black] (A) at (0,0) {};
    \node[shape=circle,draw=black] (B) at (1,0) {};
    \node[shape=circle,draw=black] (C) at (2,0) {};
    \node[shape=circle,draw=black] (D) at (3,0) {};
    \node[shape=circle,draw=black] (E) at (4,0) {};
    \node[shape=circle,draw=black] (F) at (5,0) {};
    \node[shape=circle,draw=black] (G) at (1,1) {};
    \node[shape=circle,draw=black] (H) at (2,1) {};
    \node[shape=circle,draw=black] (I) at (3,1) {};
    \node[shape=circle,draw=black] (J) at (4,1) {} ;
    \path [-] (A) edge node[left] {} (B);
    \path [-](C) edge node[right] {} (B);
    \path [-](E) edge node[right] {} (D);
    \path [-](E) edge node[left] {} (F);
    \path [-](B) edge node[right] {} (G);
    \path [-](C) edge node[right] {} (H);
    \path [-](D) edge node[left] {} (I);
    \path [-](E) edge node[right] {} (J);

\end{tikzpicture}\\
\caption{Left and Right Subgraph}
\end{figure}


Let $x$ be the number of vertices labeled $1$ on one subgraph in a friendly labeling of the whole graph. In order for the graph to be cordial $x$ cannot be $4$ or $5$ since this would make $\gamma > 3$ on the entire graph. Therefore $x$ also cannot be $0$ or $1$ since that would mean $x$ would be too large on the other sub graph.  Thus we must have $x = 2$ on one subgraph $x = 3$ on the other subgraph.\\

Given these constraints it is not possible to have a sub graph such that $\gamma = 0$, since there is no way to label two of the vertices on a subgraph without having at least two vertices of the other label connected. This means we do not need to account for the case when $\gamma=3$ on a subgraph, since the other subgraph cannot have $\gamma = 0$.
Now we will consider two cases.\\
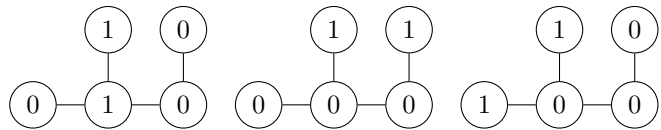
\begin{figure}[!h]
\centering
\begin{tikzpicture}
    \node[shape=circle,draw=black] (A) at (0,0) {0};
    \node[shape=circle,draw=black] (B) at (1,0) {1};
    \node[shape=circle,draw=black] (C) at (2,0) {0};
    \node[shape=circle,draw=black] (G) at (1,1) {1};
    \node[shape=circle,draw=black] (H) at (2,1) {0};
    \path [-] (A) edge node[left] {} (B);
    \path [-](C) edge node[right] {} (B);
    \path [-](B) edge node[right] {} (G);
    \path [-](C) edge node[right] {} (H);
    
    \node[shape=circle,draw=black] (one) at (3,0) {0};
    \node[shape=circle,draw=black] (two) at (4,0) {0};
    \node[shape=circle,draw=black] (three) at (5,0) {0};
    \node[shape=circle,draw=black] (four) at (4,1) {1};
    \node[shape=circle,draw=black] (five) at (5,1) {1};
    \path [-] (one) edge node[left] {} (two);
    \path [-](three) edge node[right] {} (two);
    \path [-](two) edge node[right] {} (four);
    \path [-](three) edge node[right] {} (five);
    
    \node[shape=circle,draw=black] (one) at (6,0) {1};
    \node[shape=circle,draw=black] (two) at (7,0) {0};
    \node[shape=circle,draw=black] (three) at (8,0) {0};
    \node[shape=circle,draw=black] (four) at (7,1) {1};
    \node[shape=circle,draw=black] (five) at (8,1) {0};
    \path [-] (one) edge node[left] {} (two);
    \path [-](three) edge node[right] {} (two);
    \path [-](two) edge node[right] {} (four);
    \path [-](three) edge node[right] {} (five);
\end{tikzpicture}\\
\caption{All $\gamma = 2$ subgraph labelings}
\end{figure}
\begin{figure}[!h]
\centering
\begin{tikzpicture}
    \node[shape=circle,draw=black] (A) at (3,0) {0};
    \node[shape=circle,draw=black] (B) at (1,0) {0};
    \node[shape=circle,draw=black] (C) at (2,0) {1};
    \node[shape=circle,draw=black] (G) at (1,1) {1};
    \node[shape=circle,draw=black] (H) at (2,1) {1};
    \path [-] (A) edge node[left] {} (C);
    \path [-](C) edge node[right] {} (B);
    \path [-](B) edge node[right] {} (G);
    \path [-](C) edge node[right] {} (H);
    
    \node[shape=circle,draw=black] (one) at (6,0) {1};
    \node[shape=circle,draw=black] (two) at (4,0) {0};
    \node[shape=circle,draw=black] (three) at (5,0) {0};
    \node[shape=circle,draw=black] (four) at (4,1) {1};
    \node[shape=circle,draw=black] (five) at (5,1) {1};
    \path [-] (one) edge node[left] {} (three);
    \path [-](three) edge node[right] {} (two);
    \path [-](two) edge node[right] {} (four);
    \path [-](three) edge node[right] {} (five);
    
    \node[shape=circle,draw=black] (one) at (9,0) {1};
    \node[shape=circle,draw=black] (two) at (7,0) {0};
    \node[shape=circle,draw=black] (three) at (8,0) {1};
    \node[shape=circle,draw=black] (four) at (7,1) {1};
    \node[shape=circle,draw=black] (five) at (8,1) {0};
    \path [-] (one) edge node[left] {} (three);
    \path [-](three) edge node[right] {} (two);
    \path [-](two) edge node[right] {} (four);
    \path [-](three) edge node[right] {} (five);
\end{tikzpicture}\\
\caption{All $\gamma = 1$\hbox{ subgraph labelings}}
\end{figure}
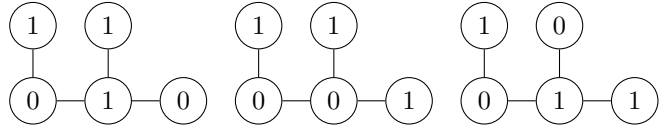
\textbf{Case 1.} With out loss of generality let the left subgraph have $\gamma = 1$, $x=2$ and the right subgraph have $\gamma = 2$, $x=3$. This would mean we would need to connect the subgraphs such that the connecting edge will not be labeled $0$. Considering all cases we see that there is no way to connect the two subgraphs to make the full graph without $\gamma = 4$ on the full graph.\\
\textbf{Case 2.} With out loss of generality let the left subgraph have $\gamma =1$, $x=2$, and the right subgraph have $\gamma=1$, $x = 3$. This would mean we need to find a way to connect the two subgraphs such that the connecting edge is labeled 0. Considering all cases shows that this is not possible.
\end{proof}

The structure of the proof leads to the following theorem.
\t\label{Theor: Zeros} Let $G \in \Gn$. We define $\Lambda (G)$ to be the number of edges, $uv$ such that vertices $u$ and $v$ have the same label for a given friendly labeling on $G$. $G$ is $(2,3)$-orientable if and only if there exsits a friendly vertex labling on $G$ such that $\Lambda (G) = \left\lceil\frac{1}{3} |E| \right \rceil$ or $\Lambda (G) = \left\lfloor \frac{1}{3} |E| \right \rfloor$, where $|E|$ is the cardinality of the edge set of $G$

\tt
\begin{proof}
Suppose $G$ satisfies $\Lambda(G) =\left\lceil\frac{1}{3} |E| \right \rceil$ or $\Lambda (G) = \left\lfloor \frac{1}{3} |E| \right \rfloor$. This would mean about $\frac{2}{3}$ of the edges are connected by vertices of different labels, and therefore arcs may be assigned such that $G$ is $(2,3)-$cordial. If $G$ is $(2,3)-$cordial, then clearly $G$ has a friendly labeling that satisfies the above conditions.
\end{proof}

Though fairly intuitive we will now use theorem 3.1 to show two more results that stem from it in the following section.
\section{Theorem 3.1 applied.}
\begin{theorem}\label{theor: gamma theorem}
The Petersen Graph, Figure \ref{fig: Petersen Graph}, is not (2,3)
-orientable. 
\end{theorem}
\begin{proof}
By Theorem \ref{Theor: Zeros} only need to show that there is no friendly labeling such that one third of the edges are connected with the same label of the vertex. Again let $x$ be the number of vertices labeled 1. We will start by dividing the Petersen graph into two subgraphs and then connect them.

\begin{figure}[h]
    \centering
    \includegraphics{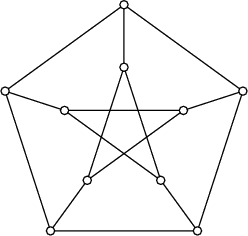}
    \caption{The Petersen Graph}
    \label{fig: Petersen Graph}
\end{figure}

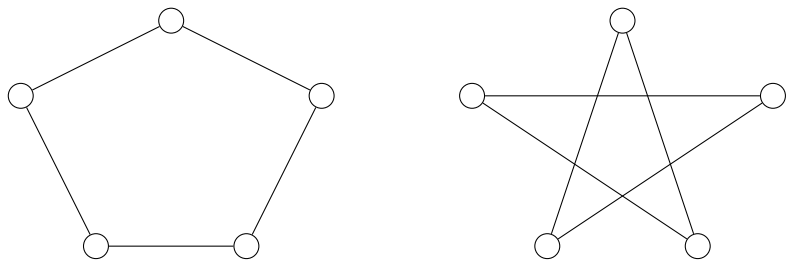
\begin{figure}[!h]
\centering
\begin{tikzpicture}
    \node[shape=circle,draw=black] (A) at (0,2) {};
    \node[shape=circle,draw=black] (B) at (-2,1) {};
    \node[shape=circle,draw=black] (C) at (-1,-1) {};
    \node[shape=circle,draw=black] (D) at (1,-1) {};
    \node[shape=circle,draw=black] (E) at (2,1) {};
    \path [-] (A) edge node[left] {} (B);
    \path [-] (B) edge node[left] {} (C);
    \path [-] (C) edge node[left] {} (D);
    \path [-] (D) edge node[left] {} (E);
    \path [-] (E) edge node[left] {} (A);
    
    \node[shape=circle,draw=black] (A) at (6,2) {};
    \node[shape=circle,draw=black] (B) at (4,1) {};
    \node[shape=circle,draw=black] (C) at (5,-1) {};
    \node[shape=circle,draw=black] (D) at (7,-1) {};
    \node[shape=circle,draw=black] (E) at (8,1) {};
    \path [-] (A) edge node[left] {} (C);
    \path[-] (A) edge node[left] {} (D);
    \path [-] (E) edge node[left] {} (B);
    \path [-] (E) edge node [left] {} (C);
    \path [-] (B) edge node [left] {} (D);
\end{tikzpicture}
\caption{Two Subgraphs of the Petersen Graph}
\end{figure}
With out loss of generality the only way for our friendly label to be cordial we must have one sub graph have $x = 2$, and have $x = 3$ on the other. The case when either sub graph has $x \geq 4$ of a certain label would result in $\gamma$ being too large. Let us first consider the outer subgraph letting $x = 3$ for this subgraph.\\

\begin{figure}[!h]
\centering
\begin{tikzpicture}
    \node[shape=circle,draw=black] (A) at (0,2) {1};
    \node[shape=circle,draw=black] (B) at (-2,1) {1};
    \node[shape=circle,draw=black] (C) at (-1,-1) {0};
    \node[shape=circle,draw=black] (D) at (1,-1) {1};
    \node[shape=circle,draw=black] (E) at (2,1) {0};
    \path [-] (A) edge node[left] {} (B);
    \path [-] (B) edge node[left] {} (C);
    \path [-] (C) edge node[left] {} (D);
    \path [-] (D) edge node[left] {} (E);
    \path [-] (E) edge node[left] {} (A);
    
    \node[shape=circle,draw=black] (A) at (6,2) {0};
    \node[shape=circle,draw=black] (B) at (4,1) {0};
    \node[shape=circle,draw=black] (C) at (5,-1) {1};
    \node[shape=circle,draw=black] (D) at (7,-1) {1};
    \node[shape=circle,draw=black] (E) at (8,1) {1};
    \path [-] (A) edge node[left] {} (B);
    \path [-] (B) edge node[left] {} (C);
    \path [-] (C) edge node[left] {} (D);
    \path [-] (D) edge node[left] {} (E);
    \path [-] (E) edge node[left] {} (A);
\end{tikzpicture}\\
\caption{$\gamma=1,\gamma=3$}
\label{fig: petersen subs}
\end{figure}
There are two cases as shown in Figure \ref{fig: petersen inner}, up to isomorphism. $\gamma = 2$ is not possible since there is no way to label two 0's and two 1's without the third 1 connecting to another vertex labeled one. This means we need to connect the star sub graph such exactly two or exactly four more edges with zero are produced.\\
\begin{figure}[!h]
\centering
\begin{tikzpicture}
    \node[shape=circle,draw=black] (A) at (0,2) {1};
    \node[shape=circle,draw=black] (B) at (-2,1) {1};
    \node[shape=circle,draw=black] (C) at (-1,-1) {0};
    \node[shape=circle,draw=black] (D) at (1,-1) {0};
    \node[shape=circle,draw=black] (E) at (2,1) {0};
    \path [-] (A) edge node[left] {} (C);
    \path[-] (A) edge node[left] {} (D);
    \path [-] (E) edge node[left] {} (B);
    \path [-] (E) edge node [left] {} (C);
    \path [-] (B) edge node [left] {} (D);
    
    \node[shape=circle,draw=black] (A) at (6,2) {1};
    \node[shape=circle,draw=black] (B) at (4,1) {0};
    \node[shape=circle,draw=black] (C) at (5,-1) {1};
    \node[shape=circle,draw=black] (D) at (7,-1) {0};
    \node[shape=circle,draw=black] (E) at (8,1) {0};
    \path [-] (A) edge node[left] {} (C);
    \path[-] (A) edge node[left] {} (D);
    \path [-] (E) edge node[left] {} (B);
    \path [-] (E) edge node [left] {} (C);
    \path [-] (B) edge node [left] {} (D);
\end{tikzpicture}\\
\caption{$\gamma=1, \gamma=3$}
\label{fig: petersen inner}
\end{figure}
Figure 8 represents all labelings of the inner subgraph up to isomorphism, with $\gamma = 1$ or $\gamma = 3$ possible on the subgraph. Upon inspection there is no possible way to connect any rotation of  the inner sub graph with either of the outer sub graphs such that $\gamma = 5$ for the resulting Petersen graph. Therefore by Theorem \ref{theor: gamma theorem} the Petersen graph is not $(2,3)$ cordial.
\end{proof}
One more example shows how this theorem proved useful in proving an upper bound on the size of the edge set for any graph to be $(2,3)$-orientable. For work on the upper bounds for other graph labelings see \cite{M}.
\t Given a directed graph $G = (V,E)$ with vertex set $V$ and $n = |V|$ with $n \geq 6$, and edge set $E$. The maximum size of $E$ such that $G$ is $(2,3)$ orientable for any given $n$ is
\tt
\begin{align}
    |E|_{max} &= \displaystyle{n\choose2} - Z + \left\lceil \frac{1}{2}\left( \displaystyle{n\choose2} - Z \right)  \right\rceil\\\label{eqn: Z eqn}
    Z &= \displaystyle{\lceil \frac{n}{2}  \rceil\choose 2} + \displaystyle{\lfloor \frac{n}{2}  \rfloor\choose 2}. 
\end{align}
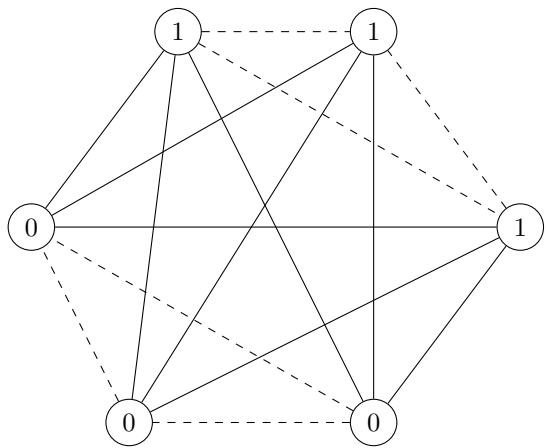
\begin{figure}
    \centering
    \begin{tikzpicture}[scale=0.65]
     \node[shape=circle,draw=black] (A) at (2,-4) {0};
     \node[shape=circle,draw=black] (B) at (-3,-4) {0};
     \node[shape=circle,draw=black] (C) at (-5,0) {0};
     \node[shape=circle,draw=black] (D) at (-2,4) {1};
     \node[shape=circle,draw=black] (E) at (2,4) {1};
     \node[shape=circle,draw=black] (F) at (5,0) {1};
     \draw[dashed] (A) -- (B);
     \draw[dashed] (A) -- (C);
     \path [-] (A) edge node[left] {} (D);
     \path [-] (A) edge node[left] {} (E);
     \path [-] (A) edge node[left] {} (F);
     
     \draw[dashed] (B) -- (C);
     \path [-] (B) edge node[left] {} (D);
     \path [-] (B) edge node[left] {} (E);
     \path [-] (B) edge node[left] {} (F);
     
     \path [-] (C) edge node[left] {} (D);
     \path [-] (C) edge node[left] {} (E);
     \path [-] (C) edge node[left] {} (F);
     
     \draw[dashed] (D) -- (E);
     \draw[dashed] (D) -- (F);
     
     \draw[dashed] (F) -- (E);
\end{tikzpicture}
\caption{A complete graph. Dashed Lines represent edges labeled zero regardless of arc orientation}
\label{fig: ghost tournament}
\end{figure}
\begin{proof}
It can be shown, see \cite {B2},  that any tournament with $n \leq 5$
 vertices is $(2,3)$-cordial, save for the case when $n =4$. Thus we begin with a complete graph with $n \geq 6$. Recall that the number of vertices on a complete graph is $n\choose2$. Thus $Z$ is the number of edges in two cliques comprised of the the subgraph of the vertices labeled $0$ and the vertices labeled $1$. In this way $Z$ counts the number of edges labeled zero regardless of arc orientation. If $n$ is even that will mean that $Z = 2 {\frac{n}{2}\choose2}$. If $n$ is odd then $Z$ is as above. This also implies there are ${n\choose2} - Z$ edges that cannot be labeled zero.
 
For every complete graph with $n \geq 6$ vertices $Z > 1 / 3 {n\choose 2}$. By Theorem \ref{Theor: Zeros} this is too many edges to be $(2,3)-$orientable. The number edges labeled $0$ must be balanced with the number of edges labeled $1$ and $-1$, and the number of edges labeled $1$ or $-1$ is $1/2({n\choose 2} - Z)$. Thus the upper bound on the number of edges a graph on $n$ vertices can have and be $(2,3)-$orientable is as in equation \ref{eqn: Z eqn}.
\end{proof}

\end{document}